\documentclass{article}
\usepackage[english]{babel}
\usepackage{subcaption}
\usepackage{amsmath, amssymb, setspace}
\usepackage{amsthm}
\usepackage[english]{babel}
\usepackage{amsfonts}
\usepackage[colorlinks,citecolor=blue,urlcolor=blue,bookmarks=false,hypertexnames=true]{hyperref}
\hypersetup{
    colorlinks=true,
    linkcolor=blue,
    filecolor=magenta,
    urlcolor=cyan,
}
\usepackage{graphics}
\usepackage[numbers,sort&compress]{natbib}
\usepackage{subcaption}
\usepackage{multicol}
\date{}
\usepackage{longtable}
\usepackage{float}
\usepackage[margin=0.7 in]{geometry}
\usepackage{graphicx}
\usepackage{mathtools}

\newtheorem{definition}{Definition}
\theoremstyle{plain}
\theoremstyle{definition}
\theoremstyle{remark}
\newtheorem{theorem}{Theorem}

\newtheorem{remark}{Remark}
\newtheorem{example}{Example}

\title{The Fractal Lie Derivative: Theory and Applications }
\author{Alireza Khalili Golmankhaneh $^{1,2}$ \thanks{Corresponding author:alirezakhalili@yyu.edu.tr, (or) alireza.khalili@iau.ac.ir (or) alirezakhalili2002@yahoo.co.in}  \\
$^1$ Department of Mathematics, Faculty of Sciences,\\ Van Yuzuncu Yil University, Van, 65080-Campus,  Turkey\\
$^2$ Department of Physics, Ur.C., Islamic Azad University,\\ Urmia 63896, West Azerbaijan, Iran\\
Elham Hashemzadeh $^{3}$\\
$^{3}$ Department of Mathematics, College of Sciences,\\ Urmia University, P.O. Box 5756151818,\\ Urmia, West Azerbaijan, Iran\\
elhamhashemzadeh69@gmail.com\\
Carlo Cattani $^{4,5}$\\
$^4$ Engineering School (DEIM), University of Tuscia,\\ 01100 Viterbo, Italy\\ $^5$ Department of Mathematics and Informatics,\\ Azerbaijan University, Baku AZ1014, Azerbaijan\\
 cattani@unitus.it\\
 Donal O'Regan  $^{6}$\\
 $^{6}$ School of Mathematical and Statistical Sciences,\\ University of Galway, Ireland\\
 donal.oregan@nuigalway.ie\\
Palle E. T. J{\o}rgensen  $^7$\\
$^7$ Department of Mathematics, The University of Iowa, \\Iowa City, IA
52242-1419, USA\\
palle-jorgensen@uiowa.edu\\
}

\date{\today}
\begin{document}

\maketitle

\begin{abstract}
This paper presents a new Lie theoretic approach to fractal calculus, which in turn yields such new results as a Fractal Noether's Theorem, a setting for fractal differential forms, for vector fields, and Lie derivatives, as well as k-fractal jet space, and algorithms for k-th fractal prolongation. The symmetries of the fractal nonlinear \(n\)-th \(\alpha\)-order differential equation are examined, followed by a discussion of the symmetries of the fractal linear \(n\)-th \(\alpha\)-order differential equation. Additionally, the symmetries of the fractal linear first \(\alpha\)-order differential equation are derived. Several examples are provided to illustrated and highlight the details of these concepts.
\end{abstract}
\textbf{Keywords:} Fractal calculus,  Fractal Lie derivatives, k-fractal jet space,  Fractal Symmetries,\\
\textbf{Mathematical Sciences Classification:}   28A80,35B06,35A30

\section{Introduction}
Fractal geometry has been used widely to investigate and characterize recurring patterns found in natural phenomena, such as blood arteries, mountains, clouds, and coastlines. The special characteristics and measurements of these formations are investigated in this topic. Fractional dimensions, self-similarity, and fractal dimensions larger than their topological dimensions are common characteristics of fractals \cite{Mandelbro,Qaswet,falconer1999techniques,feder2013fractals}.

Fractal analysis methods include harmonic analysis \cite{kigami2001analysis,MR3838440,freiberg2002harmonic, Tarasovbook2, jorgensen2006analysis, Withers}, fractional calculus, stochastic processes \cite{Barlowqq1}, and fractional space \cite{stillinger1977axiomaticq}.

Numerous Riemann-type integrals have been constructed by reworking the Fundamental Theorem of Calculus, such as \(s\)-Riemann, \(s\)-HK, and \(s\)-first-return integrals. Scholarly investigations \cite{jiang1998some,Bongiorno2018derivatives, Bongiorno2015fundamental,Bongiorno2015integral} have investigated the relationships between these integrals and the Hausdorff measure-related Lebesgue integral.
Fractals, with their complex geometric structures and self-similarity across scales \cite{juraev2024fractals}, often require methods beyond standard mathematics.
While the  literature on fractal analysis is vast, for our present discussion we stress the following papers offering new tools for Fourier expansions, and wavelets, on certain classes of fractals \cite{dutkay2006wavelets,jorgensen1998dense}.

By expanding classical calculus, a new framework for fractal calculus has been constructed, providing an algorithmic, geometric, and physically meaningful method. Functions with fractal properties, such Koch curves and Cantor sets, can be analyzed thanks to this generalization \cite{parvate2009calculus, GangalConju, satin2013fokker, Alireza-book}.

 Similar to how ordinary local calculus is expanded through Riemann-Liouville and as Caputo generalizations, non-local fractal calculus has been introduced for simulating incompressible viscous fluids in fractal media and processes with memory \cite{khalili2024fractaldd, golmankhaneh2016non, banchuin2022noise, banchuin20224noise}.
The central limit theorem and locality have been preserved in the research of sub-and super-diffusion utilizing fractal local derivatives \cite{golmankhaneh2018sub, Alireza-Fernandez-1}.

By utilizing fractal calculus in physics, new models of fractal space and time have been created \cite{golmankhaneh2021equilibrium, Tosatti, e25071008, Shlesinger-6, Vrobel-3, Welch-5, nottale1993fractal}. These models frequently have solutions that are power law and self-similarity described.

Fractal Laplace, Sumudu, and Fourier transformations applied to Cantor sets and fractal curves have been used to study the dynamics of systems with fractal times \cite{khalili2024fractalawee, khalili2021laplace, golmankhaneh2019sumudu, Fourier1ttttttt, khalili2019fractalcat, golmankhaneh2018fractalt}.
 The derivatives and integrals of many functions, such as the Weierstrass function, have also been calculated using fractal calculus \cite{gowrisankar2021fractal}.
The specification of derivative structures and bounds on fractal curves has been made possible by nonstandard analysis that incorporates hyperreal and hyperinteger integers. Additionally, fractal integral and differential forms have been defined using this method \cite{khalili2023non}. Fractal calculus is now applicable to unbounded functions thanks to gauge function generalizations \cite{golmankhaneh2016fractalgag}. Conditions for guaranteeing distinct and stable solutions have been established by studies on the stability of fractal differential equations \cite{khalili2021hyers}.

An analogue of Noether's Theorem on fractal sets has been presented, along with a suggested corresponding conservative quantity \cite{khalili2019analogues}.

A framework had been provided in this paper for understanding and analyzing non-differentiable fractal manifolds. Specialized mathematical concepts and equations, such as the Metric Tensor, Curvature Tensors, Analogue Arc Length, and Inner Product, had been introduced to enable the study of complex patterns exhibiting self-similarity across various scales and dimensions \cite{golmankhaneh2023einstein,golmankhaneh2024expansion}.

 Fractal retarded, neutral, and renewal delay differential equations with constant coefficients have been solved using the steps technique and the Laplace transform \cite{golmankhaneh2023initial}. Moreover, a new framework that incorporates the mean square derivative, fractal mean square integral, fractal mean square continuity, and the order of random variables on the fractal curve has been introduced \cite{KhaliliGolmankhanehWelchSerpaStamova2024} and generalizes mean square calculus.

A thorough and comprehensive treatment of Lie groups of transformations and their various applications for solving ordinary and partial differential equations has been provided  \cite{bluman2013symmetries,gonzalez1983symmetries}. The systematic identification of both local and nonlocal symmetries, as well as local and nonlocal conservation laws for a given partial differential equation system, has been provided, along with the systematic application of these symmetries and conservation laws for various related applications \cite{bluman2010applications}.
Results on the application of classical Lie point symmetries to problems in fluid draining, meteorology, and the epidemiology of AIDS have been established \cite{nucci1997role}.

The application of group analysis to integro-differential equations in an accessible manner has been explored. This approach has been crafted to benefit both physicists and mathematicians interested in utilizing general methods to investigate nonlinear problems through symmetries \cite{grigoriev2010symmetries,paliathanasis2015symmetries,bluman2008symmetry}.

 Approximate symmetry in nonlinear physical problems has been examined, and complex methods for Lie symmetry analysis have been discussed. Lie group classification, symmetry analysis, and conservation laws have been presented, while conservative difference schemes have been explored \cite{luo2021symmetries}.

Lie group and algebra theory has thrived and been widely applied, with much of modern elementary particle physics essentially being a physical interpretation of group theory. However, the application of symmetry methods to differential equations has remained overlooked \cite{gaeta2012nonlinear}.

 The key notions and results necessary for finding higher symmetries and conservation laws for general systems of partial differential equations have been presented \cite{vinogradov2012symmetries}.

 Also an accessible approach  has presented  to the use of symmetry methods in solving both ordinary differential equations (ODEs) and partial differential equations (PDEs) \cite{arrigo2015symmetry,hydon2000symmetries,olver1992internal,oliveri2004lie,gonzalez1988symmetries,yates2009structural,miller1973symmetries,cicogna2001partial}.

 The concept of the complete symmetry group of a differential equation has been extended to integrals of such equations. The algebras of the symmetries of both differential equations and their integrals have been studied in the context of equations represented by point or contact symmetries, ensuring clarity regarding the group. It has been found that both algebras and groups are nonunique \cite{starrett2007solving,gracia2002symmetries}.
Lie group theory has been applied to differential equations that serve as mathematical models in financial problems \cite{gazizov1998lie}.

The primary objective of this paper is to present fractal Lie derivatives and explore its applications. The structure of the paper is as follows: \\
Section \ref{S-1} presents a concise review of fractal calculus. In Section \ref{S-2}, we present  fractal Noether's Theorem. Fractal differential forms, vector fields, Lie derivatives, the \(k\)-fractal jet space, and the \(k\)-th fractal prolongation are introduced and defined in Section \ref{S-3}. Section \ref{S-4} explores the symmetries of the fractal nonlinear \(n\)-th \(\alpha\)-order differential equation, while Section \ref{S-5} discusses the symmetries of the fractal linear \(n\)-th \(\alpha\)-order differential equation. In Section \ref{S-6}, we derive the symmetries of the fractal linear first \(\alpha\)-order differential equation. Finally, Section \ref{S-7} offers concluding remarks.

\section{Fundamental Definitions in Fractal Calculus
   \label{S-1}}

In this section, we present an overview of fractal calculus as it relates to the Cantor set \( F \subset [c, d] \subset \mathbb{R} \), drawing on insights from \cite{parvate2009calculus,Alireza-book}.
\begin{definition}
The indicator function \(\mathbb{I}_{F}(K)\) for the set \(F\) is defined by
\[
\mathbb{I}_{F}(K) =
\begin{cases}
    1, & \text{if } F \cap K \neq \emptyset, \\
    0, & \text{otherwise},
\end{cases}
\]
where \( K = [c, d] \subset \mathbb{R} \).
\end{definition}

\begin{definition}
The coarse-grained measure \(\mu_{\epsilon}^{\gamma}(F, c, d)\) of \( F \cap [c, d] \) is given by
\[
\mu_{\epsilon}^{\gamma}(F, c, d) = \inf_{|\mathcal{P}| \leq \epsilon} \sum_{i=0}^{m-1} \Gamma(\gamma+1) (z_{i+1} - z_i)^{\gamma} \mathbb{I}_{F}([z_i, z_{i+1}]),
\]
where \( |\mathcal{P}| = \max_{0 \leq i \leq m-1} (z_{i+1} - z_i) \), \( \mathcal{P}_{[c, d]}=\{z_{0}=c,z_{1},...,z_{m}=d\} \), \( 0 < \gamma \leq 1 \), and \(\Gamma(\cdot)\) denotes the Gamma function.
\end{definition}

\begin{definition}
The measure function \(\mu^{\gamma}(F, c, d)\) for \( F \) is defined as
\[
\mu^{\gamma}(F, c, d) = \lim_{\epsilon \rightarrow 0} \mu_{\epsilon}^{\gamma}(F, c, d).
\]
\end{definition}
Some properties of the measure function are:
\begin{enumerate}
  \item \(\mu^{\gamma}(F + \epsilon, c + \epsilon, d + \epsilon) = \mu^{\gamma}(F, c, d)\). \\This property states that shifting the variables \(F\), \(c\), and \(d\) by a common increment \(\epsilon\) does not change the value of \(\mu^{\gamma}\).

  \item \(\mu^{\gamma}(\epsilon F, \epsilon c, \epsilon d) = \epsilon^{\gamma} \mu^{\gamma}(F, c, d)\).\\ This property shows that scaling \(F\), \(c\), and \(d\) by a factor of \(\epsilon\) results in scaling \(\mu^{\gamma}\) by \(\epsilon^{\gamma}\).
\end{enumerate}

\begin{definition}\label{iiiitttt}
The fractal \(\nu\)-dimension of \( F \cap [c, d] \) is defined by
\begin{equation}\label{w342lmv}
\begin{split}
\dim_{\nu}(F \cap [c, d]) &= \inf\{\gamma : \mu^{\gamma}(F, c, d) = 0\}\\
&= \sup\{\gamma : \mu^{\gamma}(F, c, d) = \infty\}.
\end{split}
\end{equation}
\end{definition}

\begin{definition}
The integral staircase function \( S_{F}^{\gamma}(z) \) is defined as
\[
S_{F}^{\gamma}(z) =
\begin{cases}
    \mu^{\gamma}(F, c_0, z), & \text{if } z \geq c_0, \\
    -\mu^{\gamma}(F, z, c_0), & \text{otherwise},
\end{cases}
\]
where \( c_0 \in \mathbb{R} \) is a fixed constant.
\end{definition}

\begin{definition}
The characteristic function \(\chi_{F}(z)\) on a fractal set \(F\) is defined by
\begin{equation}
  \chi_{F}(z)=\left\{
                \begin{array}{ll}
                  \Gamma(\gamma+1), & \text{if } z \in F, \\
                  0, & \text{otherwise}.
                \end{array}
              \right.
\end{equation}
Here, \(\gamma\) denotes the fractal dimension as defined in \eqref{w342lmv}.
\end{definition}

\begin{definition}
For a function \( k: F \rightarrow \mathbb{R} \), the \( F_{-}\text{lim} \) of \( k(z) \) at a point \( z \in F \) is the number \(\ell\) such that for every \( \epsilon > 0 \), there exists a \( \delta > 0 \) satisfying
\[
y \in F \text{ and } |y - z| < \delta \implies |k(y) - \ell| < \epsilon.
\]
If such an \(\ell\) exists, it is denoted by
\[
\ell = \underset{ y\rightarrow z}{F_{-}\text{lim}}~ k(y).
\]
This definition does not involve the function values at \( y \) if \( y \notin F \).
\end{definition}

\begin{definition}
A function \( k: F \rightarrow \mathbb{R} \) is said to be \( F \)-continuous at \( z \in F \) if
\begin{equation}
k(z) = \underset{ y\rightarrow z}{F_{-}\text{lim}}~ k(y),
\end{equation}
whenever the \( F_{-}\text{lim} \) exists.
\end{definition}

\begin{definition}
For a function \( k \) on a \(\gamma\)-perfect fractal set \( F \), the \( F^{\gamma} \)-derivative of \( k \) at \( z \) is defined as
\[
D_{F}^{\gamma}k(z) =
\begin{cases}
    \underset{y \rightarrow z}{F_{-}\text{lim}} \frac{k(y) - k(z)}{S_{F}^{\gamma}(y) - S_{F}^{\gamma}(z)}, & \text{if } z \in F, \\
    0, & \text{otherwise},
\end{cases}
\]
provided that the \( F_{-}\text{lim} \) exists.
\end{definition}

\begin{definition}
The \( F^{\gamma} \)-integral of a bounded function \( k(z) \), where \( k \in B(F) \) (i.e., \( k \) is bounded on \( F \)), is defined as
\begin{align}
\int_{c}^{d} k(z) \, d_{F}^{\gamma}z &= \sup_{\mathcal{P}_{[c, d]}} \sum_{i=0}^{m-1} \inf_{z \in F \cap K} k(z) \, (S_{F}^{\gamma}(z_{i+1}) - S_{F}^{\gamma}(z_i)) \\&= \inf_{\mathcal{P}_{[c, d]}} \sum_{i=0}^{m-1} \sup_{z \in F \cap K} k(z) \, (S_{F}^{\gamma}(z_{i+1}) - S_{F}^{\gamma}(z_i)),
\end{align}
where \( z \in F \), and the infimum or supremum is taken over all partitions \( \mathcal{P}_{[c, d]} \) \cite{parvate2009calculus,Alireza-book}.
\end{definition}
\section{Fractal Noether's Theorem \label{S-2}}
In this section, we explore generalized Noether's Theorem for fractal sets. The Fractal Noether's Theorem offers a rigorous mathematical connection between fractal symmetries and conservation laws in systems described by fractal Lagrangian mechanics \cite{khalili2019analogues}.
\begin{theorem}
  Consider a system with a fractal Lagrangian \( L(q_i, D_{F}^{\alpha}q_i, t) \), where \( q_i \) are the generalized coordinates, \( D_{F}^{\alpha}q_i \) are their fractal time derivatives, and \( t \) represents fractal time. The action \( S \) of the system is defined by:

\[
S = \int_{t_1}^{t_2} L(q_i, D_{F}^{\alpha}q_i, t) \, d_{F}^{\alpha}t.
\]
This Fractal Noether's theorem states that if the action \( S \) remains invariant under a fractal continuous transformation given by:

\[
q_i \to q_i' = q_i + \epsilon \eta_i(q, D_{F}^{\alpha}q_i, t),
\]

where \( \epsilon \) is a small parameter and \( \eta_i \) are the fractal infinitesimal generators of the transformation, then there exists a conserved quantity associated with this fractal symmetry:

\[
J = \frac{\partial L}{\partial (D_{F}^{\alpha}q_i)} \eta_i - f,
\]

where \( f \) is a function whose fractal time derivative can be expressed in a form that cancels out when fractal integration by parts is applied. The fractal conservation law is expressed as:

\[
D_{F}^{\alpha}J = 0.
\]

This implies that the quantity \( J \) is conserved, corresponding to the conservation law associated with the fractal symmetry of the system.
\end{theorem}

\begin{proof}
 For an \( F \)-continuous symmetry, the action \( S \) remains invariant under the transformation. This implies that the first-order variation in the action with respect to the small parameter \( \epsilon \) is zero:

\begin{equation}\label{eq:variation}
   \delta S = \int_{t_1}^{t_2} \left( \frac{\partial L}{\partial q_i} \eta_i + \frac{\partial L}{\partial D_{F}^{\alpha}q_i} D_{F}^{\alpha} \eta_i \right) d_{F}^{\alpha}t = 0.
\end{equation}

The dynamics of the system are governed by the fractal Euler-Lagrange equations \cite{golmankhaneh2025fractalvar}:

\begin{equation}\label{eq:fractaleulerlagrange}
   D_{F}^{\alpha} \left( \frac{\partial L}{\partial D_{F}^{\alpha}q_i} \right) - \frac{\partial L}{\partial q_i} = 0.
\end{equation}

Using Eqs. \eqref{eq:variation} and \eqref{eq:fractaleulerlagrange}, and applying fractal integration by parts, the variation in the action can be rewritten as:

\[
   \delta S = \epsilon \left[ \frac{\partial L}{\partial D_{F}^{\alpha}q_i} \eta_i \right]_{t_1}^{t_2} + \epsilon \int_{t_1}^{t_2} D_{F}^{\alpha} \left( \frac{\partial L}{\partial D_{F}^{\alpha}q_i} \eta_i - L \frac{d_{F}^{\alpha} \eta_i}{d_{F}^{\alpha} \dot{q}_i} \right) d_{F}^{\alpha}t = 0,
\]

where \(\dot{q}_i = D_{F}^{\alpha}q_i\), and the differential of \(\eta_i\) in fractal time is given by:

\begin{equation}\label{eq:diff_eta}
  d_{F}^{\alpha} \eta_i = \frac{\partial \eta_i}{\partial D_{F}^{\alpha}q_i} d_{F}^{\alpha} \dot{q}_i + \frac{\partial \eta_i}{\partial q_i} d_{F}^{\alpha}q_i + D_{F}^{\alpha} \eta_i \, d_{F}^{\alpha}t.
\end{equation}
Since the action is invariant under the transformation, the integrand must be a total fractal time derivative, leading to the conserved Noether current:
\[
J =  \frac{\partial L}{\partial D_{F}^{\alpha}q_i} \eta_i - f,
\]

where the function \( f \) is defined as:

\begin{equation}\label{eq:def_f}
  f = L \frac{d_{F}^{\alpha} \eta_i}{d_{F}^{\alpha} \dot{q}_i}.
\end{equation}
This establishes the conserved quantity associated with the fractal symmetry of the system, completing the proof.
\end{proof}

\begin{example}
For a system with a Lagrangian \( L(q_i, D_{F}^{\alpha}q_i, t) \) that is invariant under spatial translations, specifically under the transformation \( q_i \to q_i' = q_i + \epsilon \), where \( \epsilon \) is a small parameter representing the infinitesimal translation, the corresponding conserved quantity is the total linear momentum.

To derive this, we start by recognizing that the invariance under spatial translations implies \( \eta_i = 1 \). Substituting this into the expression for the conserved quantity \( J \) from Noether's theorem, we have:

\[
J = \frac{\partial L}{\partial D_{F}^{\alpha}q_i} \eta_i - f.
\]

With \( \eta_i = 1 \), the expression simplifies to:

\[
J = \frac{\partial L}{\partial D_{F}^{\alpha}q_i} - f.
\]

In this case, since the Lagrangian does not explicitly depend on the spatial coordinates \( q_i \), the term \( f \) vanishes because no additional terms arise from the integration by parts. Thus, the conserved quantity reduces to:

\[
J = \frac{\partial L}{\partial D_{F}^{\alpha}q_i}.
\]

This quantity \( J \) represents the generalized momentum associated with each coordinate \( q_i \).

Considering a classical system where the Lagrangian takes the form:

\[
L = \frac{1}{2} m (D_{F}^{\alpha}q_i)^2 - V(q_i, t),
\]

with \( m \) as the mass and \( V(q_i, t) \) as the potential energy, the generalized momentum is given by:

\[
\frac{\partial L}{\partial D_{F}^{\alpha}q_i} = m D_{F}^{\alpha}q_i.
\]

Therefore, the conserved quantity \( J \) corresponds directly to the classical momentum, which can be expressed as \( p_i = m D_{F}^{\alpha}q_i \). This mathematical framework succinctly captures the essence of Noether's theorem: each $F$-continuous symmetry corresponds to a specific conservation law. In this instance, the symmetry under fractal spatial translations leads to the conservation of fractal linear momentum.
\end{example}

\section{Fractal Differential Forms, Vector Fields, and Lie Derivatives \label{S-3}}
In this section, we introduce Fractal Differential Forms, Vector Fields, and Lie Derivatives. A \(k\)-form is a type of differential form that extends the concepts of functions, vectors, and higher $\alpha$-dimensional objects from multivariable calculus to higher fractal dimensions. It is an antisymmetric tensor that can be integrated over a \(k\)-dimensional fractal surface.
\begin{definition}
A fractal multivariable scalar function is a function defined as
\begin{equation}\label{12}
  f: F^n \rightarrow \mathbb{R},
\end{equation}
where \(F^n = F \times F \times \cdots \times F\) (the Cartesian product of \(n\) copies of the fractal set \(F\)), and the function maps points in this fractal space to real values.
\end{definition}
\begin{example}
Consider a fractal multivariable scalar function defined over the Cartesian product of two copies of the Cantor set:
\[
f(x, y) = \sin(x)  \cos(y),
\]
where \(x, y \in F^2\), and \(F^2 = F \times F\) represents the Cartesian product of two copies of the fractal set \(F\).

The function maps each point \((x, y)\) in this fractal space to a real value by applying the sine function to \(x\) and the cosine function to \(y\). This is a simple example of how fractal structures can be used in multivariable scalar functions.
\begin{figure}[H]
    \centering
    \includegraphics[width=0.8\textwidth]{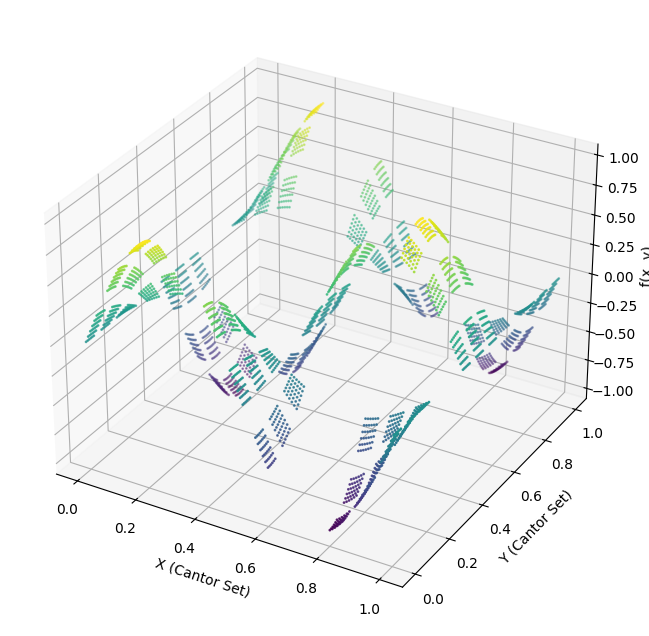}
    \caption{3D plot of the function $f(x, y) = \sin(2\pi x) \cos(2\pi y)$ evaluated on a Cantor set $\times$ Cantor set. The plot visualizes the structure of $f(x, y)$ over the fractal space formed by the Cantor set.}
    \label{fig:cantor_plot}
\end{figure}
As shown in Figure~\ref{fig:cantor_plot}, the function $f(x, y)$ exhibits fractal characteristics when evaluated on the Cantor set grid.
\end{example}
\begin{definition}
A fractal multivariable vector function is a function defined as
\begin{equation}\label{1qw2}
  X: F^{n} \rightarrow \mathbb{R}^{m},
\end{equation}
where  the function maps points in this space to vectors in \(\mathbb{R}^m\).
\end{definition}

\begin{example}

Let \( F \subset \mathbb{R} \) be the Cantor set, and consider the fractal multivariable vector function \( X: F^2 \rightarrow \mathbb{R}^2 \), defined as:
\[
X(x, y) = \begin{pmatrix} \sin(2\pi x) \\ \cos(2\pi y) \end{pmatrix},
\]
where \( (x, y) \in F^2 \), meaning both \( x \) and \( y \) are points from the Cantor set. In this example, the function maps each pair of points \( (x, y) \) from the fractal space \( F^2 \) into a vector in \( \mathbb{R}^2 \). Specifically, the first component is determined by \( \sin(2\pi x) \), and the second component by \( \cos(2\pi y) \), both periodic functions that provide a smooth, continuous map on the fractal domain.
\begin{figure}[H]
    \centering
    \includegraphics[width=0.8\textwidth]{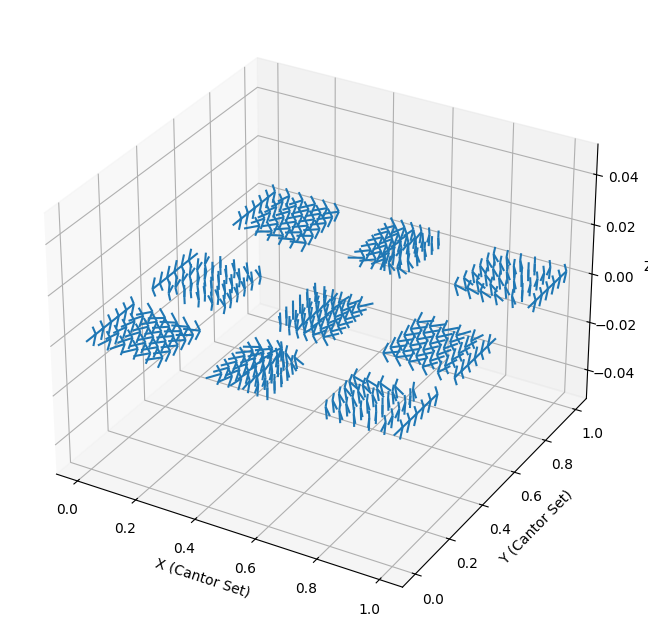} 
    \caption{The figure shows the projection of the vector field, with the x-component of the field represented as \( U = \sin(2\pi X) \) and the y-component as \( V = \cos(2\pi Y) \), where \( X \) and \( Y \) are points from the Cantor set.
    }
    \label{fig:cantor_vector_field}
\end{figure}
In Figure we have plotted \ref{fig:cantor_vector_field} which representation offers insight into how fractal geometries interact with vector fields, which can have applications in modeling physical systems within fractal spaces.
\end{example}
\begin{definition}
A fractal vector field is defined as:
\[
  X_{F}: F^n \rightarrow \mathbb{R}^n,
\]
where  the vector field assigns vectors in \(\mathbb{R}^n\) to each point in \(F^n\).
\end{definition}
\begin{example}
Let \( F \subset \mathbb{R} \) be the Cantor set, and consider the fractal vector field \( X_F: F^2 \rightarrow \mathbb{R}^2 \) defined as:
\[
X_F(x, y) = \begin{pmatrix} x(1 - x) \\ y(1 - y) \end{pmatrix},
\]
where \( (x, y) \in F^2 \).

In this example, the fractal vector field \( X_F \) assigns a vector in \( \mathbb{R}^2 \) to each point \( (x, y) \) from the fractal space \( F^2 \). The first component of the vector at each point is given by \( x(1 - x) \), and the second component by \( y(1 - y) \). This choice ensures that the vector field remains within the bounded fractal space and demonstrates how vectors can vary smoothly over a fractal domain.
\end{example}
\begin{definition}
A fractal tensor field is a function defined as:
\[
  \mathbf{T}: F^n \rightarrow T^r_s(\mathbb{R}^m),
\]
where  the tensor field assigns a \((r, s)\)-tensor in \(\mathbb{R}^m\) to each point in \(F^n\). Here, \(T^r_s(\mathbb{R}^m)\) denotes the space of \((r, s)\)-tensors, which are multilinear maps that take \(r\) vectors and \(s\) covectors (dual vectors) and output a real number. A fractal tensor field assigns a \((r, s)\)-tensor to each point in the fractal space \(F^n\).
\end{definition}
\begin{definition}
A fractal tangent vector \(v\) at a point \(p\) on a fractal manifold \(M_F\) is a linear map that acts on fractal functions \(f: M_F \rightarrow \mathbb{R}\). It represents how the function \(f\) changes in the direction specified by \(v\). Specifically, \(v\) satisfies a generalized Leibniz rule:
\[
v(fg) = v(f)g(p) + f(p)v(g),
\]
for any fractal smooth  functions \(f\) and \(g\). The set of all fractal tangent vectors at \(p\) forms the  fractal tangent space \(T_p M_F\), capturing the local behavior of the fractal structure around \(p\).
\end{definition}
\begin{definition}
A \(k\)-fractal form $\omega_{F}$ on a fractal smooth manifold \(M_{F}\) is an object that assigns to every point on the fractal manifold. That means it takes in \(k\) fractal tangent vectors at each point on the fractal manifold and outputs a real number, satisfying the following properties:
\begin{enumerate}
  \item Linearity: A \(k\)-fractal form is linear in each of its arguments. For fractal vectors \(X_1, X_2, \dots, X_k\) and scalars \(\zeta\), \(\eta\), we have:
   \[
   \omega_{F}(\zeta\, X_1 + \eta\, Y_1, X_2, \dots, X_k) = \zeta\, \omega_{F}(X_1, X_2, \dots, X_k) + \eta\, \omega_{F}(Y_1, X_2, \dots, X_k).
   \]
  \item Antisymmetry: A \(k\)-fractal form changes sign when two of its arguments are swapped. For example:
   \[
   \omega(X_1, X_2, \dots, X_i, X_{i+1}, \dots, X_k) = -\omega(X_1, X_2, \dots, X_{i+1}, X_i, \dots, X_k).
   \]
   This implies that if two vectors in the set \(X_1, \dots, X_k\) are the same, the form evaluates to zero.
\end{enumerate}
\end{definition}
\begin{example}
 A 0-fractal form is just a fractal smooth function \(f\) on the fractal manifold \(M_{F}\). Let \( F \subset \mathbb{R} \) be the Cantor set, and consider the 0- fractal form \( f: F \rightarrow \mathbb{R} \)  defined as:
\begin{equation}\label{4312}
f(x) = \sin(2\pi x),
\end{equation}
where \( x \in F \). \\

Figure \ref{fig:cantor-0form} shows the plot of the 0-fractal form \( f(x) = \sin(2 \pi x) \) evaluated on the Cantor set. The function assigns a fractal smooth scalar value to each point in the Cantor set.
\begin{figure}[H]
    \centering
    \includegraphics[width=0.8\textwidth]{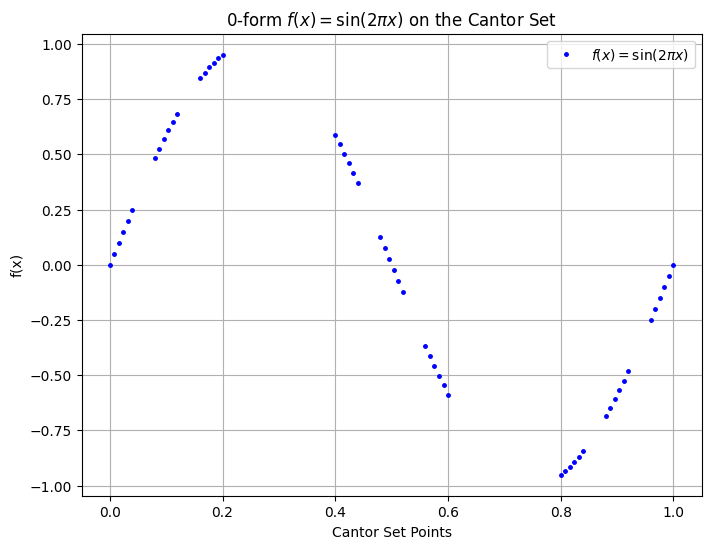} 
    \caption{Plot of the 0-fractal form Eq.\eqref{4312} on the Cantor set. The function values are represented as individual points corresponding to the Cantor set elements.}
    \label{fig:cantor-0form}
\end{figure}
\end{example}
 \begin{example}
 A 1-fractal form can be thought of as a fractal covector or a fractal dual vector. In local coordinates \(\{x^1, x^2, \dots, x^n\}\), a 1-fractal form is written as:
  \[
  \omega_{F} = \sum_{i=1}^n \omega_i(x) d_{F}^{\alpha}x^i,
  \]
  where the \(d_{F}^{\alpha}x^i\) are the fractal coordinate differentials and \(\omega_i(x):F\rightarrow \mathbb{R}\) are fractal smooth functions.
\end{example}
\begin{example} A 2-fractal form is written in local coordinates as:
  \[
  \omega_{F} = \sum_{i < j} \omega_{ij}(x) \, d_{F}^{\alpha}x^i \wedge  d_{F}^{\alpha} x^j,
  \]
  where \( d_{F}^{\alpha}x^i \wedge  d_{F}^{\alpha}x^j\) is the wedge product of the fractal differentials, and the coefficients \(\omega_{ij}(x)\) are fractal smooth functions.
\end{example}
\begin{example} A \(k\)-fractal form in local coordinates is written as:
  \[
  \omega_{F} = \sum_{i_1 < i_2 < \cdots < i_k} \omega_{i_1 i_2 \cdots i_k}(x) \, d_{F}^{\alpha}x^{i_1} \wedge d_{F}^{\alpha}x^{i_2} \wedge \cdots \wedge d_{F}^{\alpha}x^{i_k},
  \]
  where the wedge product \(\wedge\) ensures the antisymmetry property of the form. A \(k\)-fractal form can be fractal integrated over a \(k\alpha\)-dimensional oriented fractal surface or fractal manifold. For example, in the case of a 2-fractal form, one can think of it as something that can be fractal integrated over a fractal surface in $3\alpha$-dimensional space.
\end{example}
\begin{definition}
Given a differential \(k\)- fractal form \(\omega_{F}\) on a fractal smooth manifold \(M_{F}\), the fractal exterior derivative \(d_{F}^{\alpha}\omega\) is a \((k+1)\)-fractal form that satisfies the following properties:
\begin{enumerate}
  \item Linearity: \(d_{F}^{\alpha}(\omega_1 + \omega_2) = d_{F}^{\alpha}\omega_1 + d_{F}^{\alpha}\omega_2\), where \(\omega_1\) and \(\omega_2\) are \(k\)-fractal forms.
  \item Leibniz Rule: If \(\omega_1\) is a \(k\)-fractal form and \(\omega_2\) is an \(\ell\)-fractal form, then:
   \[
   d_{F}^{\alpha}(\omega_1 \wedge \omega_2) = d_{F}^{\alpha}\omega_1 \wedge \omega_2 + (-1)^k \omega_1 \wedge d_{F}^{\alpha}\omega_2.
   \]
   \item Double derivative is zero: \(d_{F}^{\alpha}(d_{F}^{\alpha}\omega_{F}) = 0\) for any fractal differential form \(\omega_{F}\). This means the fractal exterior derivative is nilpotent.
\end{enumerate}
\end{definition}
\begin{example}
In local coordinates, for a \(k\)-fractal form \(\omega_{F} = \sum_{i_1, \ldots, i_k} \omega_{i_1, \ldots, i_k} d_{F}^{\alpha}x^{i_1} \wedge \cdots \wedge d_{F}^{\alpha}x^{i_k}\), the fractal exterior derivative \(d_{F}^{\alpha}\omega_{F}\) is given by:
\[
d_{F}^{\alpha}\omega_{F} = \sum_{i_1, \ldots, i_k} D_{F,x^j}^{\alpha} \omega_{i_1, \ldots, i_k} d_{F}^{\alpha}x^j \wedge d_{F}^{\alpha}x^{i_1} \wedge \cdots \wedge d_{F}^{\alpha}x^{i_k}.
\]
The fractal exterior derivative takes a fractal differential \(k\)-fractal form and returns a \((k+1)\)-fractal form, allowing for a generalization of concepts like fractal gradient, fractal curl, and fractal divergence in fractal vector calculus.
\end{example}
\begin{definition}
Given a fractal smooth manifold, let \(X_{F}\) be a fractal vector field on the fractal manifold, and let \(f\) be a $F^{\alpha}$-differentiable function. The fractal Lie derivative of \(f\) with respect to \(X_{F}\), denoted by \(\mathcal{L}_{X_{F}} f\), fractal measures the rate of change of \(f\) along the flow generated by \(X_{F}\). For a function \(f\), the fractal Lie derivative is simply the fractal directional derivative:
\[
\mathcal{L}_{X_{F}} f = X_{F}(f).
\]
If \(X_{F} = \xi^i D_{F,x_{i}}^{\alpha}\), then:
\[
\mathcal{L}_{X_{F}} f= \xi^i D_{F,x_{i}}^{\alpha}f.
\]
\end{definition}
\begin{definition}
For a fractal vector field \(Y_{F}\), the fractal Lie derivative of  \(Y_{F}\) with respect to \(X_{F}\), \(\mathcal{L}_{X_{F}} Y_{F}\), is given by the commutator (or fractal Lie bracket) of the two fractal vector fields:
\[
\mathcal{L}_{X_{F}} Y_{F} = [X_{F}, Y_{F}] = X_{F}^i D_{F,x^{i}}^{\alpha}Y^j D_{F,x^{j}}^{\alpha} - Y_{F}^i  D_{F,x^{i}}^{\alpha}X^j  D_{F,x^{j}}^{\alpha}.
\]
\end{definition}
\begin{definition}
For a \( k \)-fractal form \( \omega_{F} \), the fractal interior product (or contraction) \( \iota_{X_{F}} \omega_{F} \) is defined as:
\[
(\iota_{X_{F}} \omega_{F})(X_1, X_2, \dots, X_{k-1}) = \omega(X, X_1, X_2, \dots, X_{k-1}),
\]
where \( X_1, X_2, \dots, X_{k-1} \) are fractal vector fields, and \( X_{F} \) is the fractal vector field being contracted. The fractal interior product \( \iota_{X_{F}} \omega_{F} \) inserts the fractal vector field \( X_{F} \) into the first slot of the \( k \)-fractal form \( \omega_{F} \), effectively reducing its degree by one.
\end{definition}
\begin{example}
For a 2-fractal form \( \omega_{F} = f(x) \, d_{F}^{\alpha}x \wedge d_{F}^{\alpha}y \), the contraction with a fractal vector field \( X_{F} = D_{F,x}^{\alpha} \) is:
\[
\iota_{X_{F}} \omega_{F} = f(x) \, d_{F}^{\alpha}y.
\]
In this example, \( X_{F} \) acts on the first component of the fractal differential form, resulting in a 1-fractal form.
\end{example}
\begin{definition}
For a fractal differential form \(\omega_{F}\), the fractal Lie derivative along a fractal vector field \(X_{F}\) is given by:
\[
\mathcal{L}_{X_{F}} \omega_{F} = d_{F}^{\alpha}(\iota_{X_{F}} \omega_{F}) + \iota_{X_{F}} (d_{F}^{\alpha}\omega_{F}).
\]
\end{definition}
\begin{definition}
Let \(M_{F}\) be a fractal smooth manifold, and let \(E_{F}=M_{F} \times \mathbb{R} \to M_{F}\) be a fractal smooth fiber bundle, where \(M_{F}\) is the fractal base space and \(E_{F}\) is the fractal total space. Given a fractal smooth section/function \(u: M_{F} \to E_{F}\), the k-fractal jet of \(u\) at a point \(p \in M_{F}\) is an equivalence class of fractal smooth sections, where two sections are equivalent if their values and the values of all their fractal derivatives up to order \(k\) agree at \(p\). The space of all such \(k\)-fractal jets at \(p\) is called the k-fractal jet space, denoted as \(J_{F}^k(M_{F}, E_{F})\).
\end{definition}
\begin{example}
For example, if \(u: M_{F} \to \mathbb{R}\) is a function and \(M_{F} = \mathbb{F}^n\), the k-fractal jet space \(J_{F}^k(M_{F}, \mathbb{R})\) contains all information about \(u(p)\) and its fractal partial derivatives up to order \(k\) at \(p\).
\end{example}

\begin{definition}
If \(M_{F}= F^n\) and \(u: M_{F} \to \mathbb{R}\), local coordinates on the fractal jet space \(J_{F}^k(M_{F}, \mathbb{R})\) are given by:
\begin{equation}
(x_1, x_2, \dots, x_n, u, D_{F,x_{1}}^{\alpha}u, D_{F,x_{2}}^{\alpha}u, \dots, D_{F,x_1 x_1}^{\alpha}u, D_{F,x_1 x_2}^{\alpha}u\dots, D_{F,x_1 \cdots x_n}^{\alpha}u),
\end{equation}
where \(x_1, x_2, \dots, x_n\) are the coordinates on the fractal base manifold \(M_{F}\), \(u\) is the fractal function value, and \(D_{F,x_{i}}^{\alpha}u\), $D_{F,x_i x_j}^{\alpha}u$, etc., represent the fractal partial derivatives of \(u\) with respect to the coordinates \(x_1, \dots, x_n\), up to order \(k\). For higher \(k\), the fractal jet space includes higher-order fractal partial derivatives of the function.
\end{definition}
 \begin{example}
 The first fractal jet space \(J_{F}^1(M_{F}, E_{F})\) contains information about the fractal function and its first $F^{\alpha}$-derivatives, which is enough to study first $\alpha$-order fractal partial differential equations (FPDEs). The second fractal jet space \(J_{F}^2(M_{F}, E_{F})\) includes second $F^{\alpha}$-derivatives, making it suitable for studying second $\alpha$-order FPDEs.
 \end{example}
\begin{definition}
Let \( X_{F} = \sum_{i=1}^{n} \xi^i(x_1, \dots, x_n) D_{F,x_{i}}^{\alpha} \) be a smooth fractal vector field on \( F^n \), where \( \xi^i \) are  smooth fractal functions, and \( x_1, \dots, x_n \) are the coordinates on \( \mathbb{R}^n \). The first fractal prolongation of the fractal vector field \( X_{F} \), denoted by \( X_{F}^{(1)} \), is a fractal vector field on the first-order fractal jet space \( J_{F}^1(F^n, \mathbb{R}) \), which includes both the coordinates \( x_1, \dots, x_n \) and the first-order $F^{\alpha}$-derivatives \( u_i = D_{F,x_{i}}^{\alpha}u \), where \( u \) is a fractal smooth function on \( F^n \). The fractal prolongation \( X^{(1)} \) is defined by:
\begin{equation}\label{r454322}
X_{F}^{(1)} = X_{F} + \sum_{i=1}^{n} \left((D_{F,x_j}^{\alpha}\xi^i) u_{i} + \xi^i (D_{F,x_j}^{\alpha} u_{i})  \right)D_{F,u_i}^{\alpha}.
\end{equation}
This process can be generalized to higher-order fractal jet spaces, leading to the k-th fractal prolongation \( X_{F}^{(k)} \), which includes higher-order $F^{\alpha}$-derivatives of the function \( u \).
\end{definition}

\section{Symmetries of the Fractal Nonlinear n-th \(\alpha\)-Order Differential Equation \label{S-4}}
In this section, we examine the symmetries of the fractal nonlinear \(n\)-th \(\alpha\)-order differential equation.\\
Consider a general fractal differential equation of the form:
\begin{equation}\label{2342sws}
\mathcal{F}(x, y, D_{F,x}^{\alpha}y, D_{F,x}^{2\alpha}y, \ldots, D_{F,x}^{n\alpha} y) = \mathcal{F}(x, y, y^{(\alpha)}, y^{(2\alpha)}, \ldots, y^{(n\alpha)}) = 0,
\end{equation}
where \(y:F\rightarrow \mathbb{R}\) is the unknown function, and \(y^{(\alpha)}, y^{(2\alpha)}, \ldots\) are the fractal derivatives of \(y\) with respect to \(x\). Let us introduce infinitesimal transformations as:
\[
(S_{F}^{\alpha}(x))^* = S_{F}^{\alpha}(x) + \epsilon \xi(x, y), \quad (S_{F}^{\alpha}(y))^* = S_{F}^{\alpha}(y) + \epsilon \phi(x, y),
\]
where \(\epsilon\) is a small parameter, and \(\xi(x, y)\) and \(\phi(x, y)\) are functions to be determined. The associated fractal infinitesimal generator is defined as:
\[
V_{F} = \xi(x, y) D_{F,x}^{\alpha} + \phi(x, y) D_{F,y}^{\alpha}.
\]

To account for the fractal derivatives of \(y\), we prolong the vector field \(V_{F}\) to \(V_{F}^{(n\alpha)}\), including terms up to the n-th \(\alpha\)-order fractal derivative in Eq.\eqref{2342sws}:
\[
V_{F}^{(n\alpha)} = V_{F} + \eta^{(1)} D_{F,y^{(\alpha)}}^{\alpha} + \eta^{(2)} D_{F,y^{(2\alpha)}}^{\alpha} + \ldots + \eta^{(n)} D_{F,y^{(n\alpha)}}^{\alpha},
\]
where the fractal prolongation coefficients \(\eta^{(k)}\) are recursively defined by:
\[
\eta^{(k)} = D_{F}^{\alpha} \eta^{(k-1)} - y^{(k\alpha)} D_{F}^{\alpha} \xi, \quad \eta^{(0)} = \phi - \xi y^{(\alpha)}.
\]
Here, the total fractal derivative \(D_{F}^{\alpha}\) is given by:
\[
D_{F}^{\alpha} = D_{F,x}^{\alpha} + y^{(\alpha)} D_{F,y}^{\alpha} + y^{(2\alpha)} D_{F,y^{(\alpha)}}^{\alpha} + \ldots.
\]
Next, applying the prolonged vector field to Eq.\eqref{2342sws}, we obtain:
\[
V^{(n\alpha)}\left(\mathcal{F}(x, y, y^{(\alpha)}, y^{(2\alpha)}, \ldots, y^{(n\alpha)})\right) = 0.
\]
This expands to:
\begin{equation}\label{67bcerd951}
\xi D_{F,x}^{\alpha}\mathcal{F} + \phi D_{F,y}^{\alpha}\mathcal{F} + \eta^{(1)}  D_{F,y^{(\alpha)}}^{\alpha}\mathcal{F} + \eta^{(2)} D_{F,y^{(2\alpha)}}^{\alpha}\mathcal{F} + \ldots + \eta^{(n)} D_{F,y^{(n\alpha)}}^{\alpha}\mathcal{F}  = 0.
\end{equation}
In this expression, the first two terms \(\xi D_{F,x}^{\alpha}\mathcal{F} + \phi D_{F,y}^{\alpha}\mathcal{F}\) represent the fractal Lie derivative, while the remaining terms involve the fractal prolongation coefficients \(\eta^{(k)}\). Eq.\eqref{67bcerd951} serves as the determining equation for the infinitesimal generators \(\xi(x, y)\) and \(\phi(x, y)\), and solving this equation reveals the symmetries of the fractal differential equation.

\section{Symmetries of the Fractal Linear n-th $\alpha$-Order Differential Equation \label{S-5}}
In this section, we explore the symmetries of the fractal linear \(n\)-th \(\alpha\)-order differential equation.\\
Consider a general linear \(n\)-th order fractal differential equation  of the form:
\begin{equation}\label{23}
\mathcal{L}(x, y, y^{(\alpha)}, y^{(2\alpha)}, \ldots, y^{(n\alpha)}) = a_n(x) y^{(n\alpha)} + a_{n-1}(x) y^{((n-1)\alpha)} + \ldots + a_1(x) y^{(\alpha)} + a_0(x) y = 0,
\end{equation}
where \(y:F\rightarrow \mathbb{R}\) is the unknown fractal function and the \(a_i:F\rightarrow \mathbb{R}\) are known fractal coefficient functions. To find the symmetries of Eq.\eqref{23}, we first define fractal infinitesimal transformations  as:
\[
(S_{F}^{\alpha}(x))^* = S_{F}^{\alpha}(x) + \epsilon \xi(x, y), \quad (S_{F}^{\alpha}(y))^* = S_{F}^{\alpha}(y) + \epsilon \phi(x, y),
\]
where \(\epsilon\) is a small parameter, and \(\xi(x, y)\) and \(\phi(x, y)\) are fractal functions to be determined. The associated fractal infinitesimal generator \(V_{F}\) is given by:
\[
V_{F} = \xi(x, y) D_{F,x}^{\alpha} + \phi(x, y) D_{F,y}^{\alpha}.
\]
We need to extend the fractal vector field \(V_{F}\) to include higher $\alpha$-order fractal derivatives of \(y\). The fractal prolonged vector field \(V_{F}^{(n\alpha)}\) up to the \(n\)-th order is given by:
\[
V_{F}^{(n\alpha)} = V_{F} + \eta^{(1)}D_{F,y^{(\alpha)}}^{\alpha} + \eta^{(2)}D_{F,y^{(2\alpha)}}^{\alpha} + \ldots + \eta^{(n)} D_{F,y^{(n\alpha)}}^{\alpha},
\]
where the fractal prolongation coefficients \(\eta^{(k)}\) are defined recursively as:
\[
\eta^{(k)} = D_{F}^{\alpha} \eta^{(k-1)} - y^{(k)} D_{F}^{\alpha} \xi, \quad \eta^{(0)} = \phi - \xi y^{\alpha}.
\]
Apply the fractal prolonged vector field \(V^{(n\alpha)}\) to Eq.\eqref{23}, we obtain:
\[
V^{(n\alpha)}\left(\mathcal{L}(x, y, y^{(\alpha)},  y^{(2\alpha)}, \ldots, y^{(n\alpha)})\right) = 0.
\]
This leads to:
\begin{equation}\label{3455mn}
\xi D_{F,x}^{\alpha} \mathcal{L} + \phi D_{F,y}^{\alpha} \mathcal{L}+ \eta^{(1)} D_{F,y^{\alpha}}^{(\alpha)} \mathcal{L} + \eta^{(2)} D_{F,y^{2\alpha}}^{(\alpha)} \mathcal{L} + \ldots + \eta^{(n)} D_{F,y^{n\alpha}}^{(\alpha)} \mathcal{L} = 0.
\end{equation}
After expanding and simplifying Eq.\eqref{3455mn}, one collect terms involving different fractal derivatives of \(y\) (i.e., \(y^{(\alpha)}, y^{(2\alpha)}, \ldots, y^{(n\alpha)}\)) and set the coefficients of each fractal derivative to zero. This produces a system of partial fractal differential equations (FPDEs) for \(\xi(x, y)\) and \(\phi(x, y)\), which is called  the fractal determining equations. Solving the fractal determining equations give  the infinitesimal generators \(\xi(x, y)\) and \(\phi(x, y)\). These solutions represent the symmetries of Eq.\eqref{23}.

\begin{example}
Consider the second $\alpha$-order linear fractal differential equation:
\begin{equation}\label{3422kkjtyzaw}
D_{F}^{2\alpha}y + p(x)D_{F}^{\alpha}y  + q(x) y = 0.
\end{equation}
The fractal infinitesimal generator is:
\[
V_{F} = \xi(x, y) D_{F,x}^{\alpha} + \phi(x, y) D_{F,y}^{\alpha}.
\]
Hence the fractal prolong the fractal vector field is:
\[
V_{F}^{(2)} = V_{F} + \eta^{(1)} D_{F,y^{(\alpha)}}^{\alpha} + \eta^{(2)} D_{F,y^{(2\alpha)}}^{\alpha},
\]
where:
\[
\eta^{(1)} = D_{F}^{\alpha} \phi - y^{(\alpha)} D_{F}^{\alpha} \xi, \quad \eta^{(2)} = D_{F}^{\alpha}  \eta^{(1)} - y^{(2\alpha)} D_{F}^{\alpha}  \xi.
\]
Applying the fractal prolonged vector field to Eq.\eqref{3422kkjtyzaw} gives:
\[
V_{F}^{(2)}\left(D_{F}^{2\alpha}y + p(x)D_{F}^{\alpha}y  + q(x) y \right) = 0.
\]
After substituting \(\eta^{(1)}\) and \(\eta^{(2)}\) into the fractal prolonged equation and expanding all terms, we collect terms involving \(y^{(2\alpha)}, y^{(\alpha)},\) and \(y\) and set the coefficients of each fractal derivative to zero. This gives us a system of fractal partial differential equations  for \(\xi(x, y)\) and \(\phi(x, y)\). The coefficient of \(y^{(2\alpha)}\) gives:
\[
\eta^{(2)} = \xi(x, y) = 0.
\]
This implies that \(\xi\) does not depend on \(y\), so \(\xi = \xi(x)\). The coefficient of \(y^{\alpha}\) gives:
\[
\eta^{(1)} + p(x) \xi(x) = 0.
\]
Substituting \(\eta^{(1)} = D_{F}^{\alpha} \phi(x, y) - y^{(\alpha)} D_{F}^{\alpha} \xi(x)\) leads to a FPDE involving \(\phi(x, y)\) and \(\xi(x)\). The coefficient of \(y\) gives:
\[
q(x) \xi(x) + D_{F}^{\alpha} \phi(x, y) = 0.
\]
This provides another condition for \(\phi(x, y)\) and \(\xi(x)\). By solving the system of FPDEs obtained from the different terms, we determine the explicit forms of \(\xi(x)\) and \(\phi(x, y)\). These solutions describe the symmetries of the second $\alpha$-order fractal differential equation.
\end{example}

\section{ Symmetries of the Fractal Linear First $\alpha$-Order Differential Equation \label{S-6}}
In this section, we investigate the symmetries of the fractal linear first \(\alpha\)-order differential equation.\\
Consider the first $\alpha$-order linear differential equation:
\begin{equation}\label{iiuoojnhy}
D_{F,x}^{\alpha}y =y^{(\alpha)}= f(x, y).
\end{equation}
 We define infinitesimal transformations as:
\begin{equation}\label{e344385}
  (S_{F}^{\alpha}(x))^* = S_{F}^{\alpha}(x) + \epsilon \xi(x, y), \quad (S_{F}^{\alpha}(y))^* = S_{F}^{\alpha}(y) + \epsilon \phi(x, y).
\end{equation}
where \(\epsilon\) is a small parameter, and \(\xi(x, y)\) and \(\phi(x, y)\) are fractal functions to be determined. The infinitesimal generator corresponding to Eq.\eqref{e344385} is:
\[
V_{F} = \xi(x, y)D_{F,x}^{\alpha} + \phi(x, y) D_{F,y}^{\alpha}.
\]
To account for the fractal derivative \(y^{(\alpha)}\), we prolong the fractal vector field \(V_{F}\) by adding a term for \(y^{(\alpha)}\). The prolonged fractal vector field \(V_{F}^{(1)}\) is:
\[
V^{(1)}_{F} = V_{F} + \eta^{(1)} D_{F,y^{\alpha}},
\]
where \(\eta^{(1)}\) is the fractal prolongation term given by:
\[
\eta^{(1)} = D_{F}^{\alpha} \phi - y^{\alpha}D_{F}^{\alpha} \xi.
\]
 Expanding this expression, we get:
\[
\eta^{(1)} = D_{F,x}^{\alpha} \phi + y^{\alpha}D_{F,y}^{\alpha}  \phi - y^{\alpha} (D_{F,x}^{\alpha} \xi + y^{\alpha} D_{F,x}^{\alpha} \xi ).
\]
Next, we apply the fractal prolonged vector field \(V_{F}^{(1)}\) to the Eq.\eqref{iiuoojnhy}. This gives:
\[
V^{(1)}_{F} \left( y^{(\alpha)} - f(x, y) \right) = 0.
\]
Expanding this equation, we get:
\[
\xi(x, y)D_{F,x}^{\alpha}  \left( y^{(\alpha)} - f(x, y) \right) + \phi(x, y) D_{F,y}^{\alpha} \left( y^{(\alpha)} - f(x, y) \right) + \eta^{(1)} D_{F,y^{(\alpha)}}^{\alpha} \left( y^{(\alpha)} - f(x, y) \right) = 0.
\]
This simplifies to:
\begin{equation}\label{123ii7uy6t5r4}
 D_{F,x}^{\alpha}  \phi +  y^{\alpha}  D_{F,y}^{\alpha}\phi - \left(D_{F,x}^{\alpha}  \xi + y^{\alpha} D_{F,y}\xi \right) y^{\alpha} - \xi D_{F,x}^{\alpha} f  - \phi D_{F,y}^{\alpha} f  = 0.
  \end{equation}

By separating Eq.\eqref{123ii7uy6t5r4} into terms involving \(y^{\alpha}\) and equate coefficients  leads to a system of fractal differential equations in terms of \(\xi(x)\) and \(\phi(x, y) =S_{F}^{\alpha}(y) D_{F}^{\alpha} \xi(x)  + h(x) \) as follows:
\begin{align}\label{qw1221}
  D_{F}^{2\alpha} \xi(x) - D_{F,y}^{\alpha}f(x,y)D_{F}^{\alpha} \xi(x) &= 0\nonumber\\
 D_{F}^{\alpha}h(x)  - \xi(x) D_{F,x}^{\alpha}f(x,y)  - h(x)D_{F,y}^{\alpha}f(x,y) &= 0
\end{align}
These fractal equations determine the form of the infinitesimals \(\xi(x)\) and \(\phi(x, y)\) depending on the function \( f(x, y) \). Solving these equations will provide the fractal symmetries of the fractal differential equation.

\begin{remark}
Once the infinitesimal generators \(\xi(x, y)\) and \(\phi(x, y)\) are determined, they represent the symmetries of the fractal  differential equation. These symmetries correspond to transformations  that leave the differential equation invariant. Such symmetries can be useful in simplifying the equation, reducing its order, or finding conserved quantities (via Noether's theorem).
\end{remark}
\begin{example}
Consider the following first \(\alpha\)-order fractal differential equation:
\begin{equation}\label{44543mn}
D_{F}^{\alpha}y = S_{F}^{\alpha}(x) + S_{F}^{\alpha}(y),
\end{equation}
where \(D_{F}^{\alpha}\) denotes the fractal derivative, and \(S_{F}^{\alpha}(x)\) and \(S_{F}^{\alpha}(y)\) are fractal functions.
Using Equation \eqref{qw1221}, we derive the system of fractal determining equations for this fractal differential equation as:
\[
D_{F}^{2\alpha}\xi(x) - D_{F}^{\alpha}\xi(x) = 0,
\]
\[
D_{F}^{\alpha}h(x) - \xi(x) - h(x) = 0.
\]
By solving the system for \(\xi(x)\) and \(h(x)\), we obtain the following solutions:
\[
\xi(x) = c_1 \exp(S_{F}^{\alpha}(x)) + c_2,
\]
\[
h(x) = c_1 S_{F}^{\alpha}(x) \exp(S_{F}^{\alpha}(x)) + c_2 + c_3 \exp(S_{F}^{\alpha}(x)),
\]
where \(c_1\), \(c_2\), and \(c_3\) are constants of integration. Thus, the fractal infinitesimals \(\xi(x)\) and \(\phi(x, y)\) are:
\[
\xi(x) = c_1 \exp(S_{F}^{\alpha}(x)) + c_2,
\]
\[
\phi(x, y) = S_{F}^{\alpha}(y) D_{F}^{\alpha} \xi(x) + h(x) = c_1 S_{F}^{\alpha}(y) \exp(S_{F}^{\alpha}(x)) + \left( c_1 S_{F}^{\alpha}(x) \exp(S_{F}^{\alpha}(x)) + c_3 \exp(S_{F}^{\alpha}(x)) + c_2 \right).
\]
Using the transformations provided by the fractal infinitesimals \(\xi(x)\) and \(\phi(x, y)\) in Equation \eqref{44543mn}, we can transform the fractal differential equation into a separable form, making it easier to solve.
\end{example}

\begin{example}
Consider the second $\alpha$-order linear fractal differential equation  as the:
\begin{equation}\label{o9i00}
D_{F}^{2\alpha}y + y = 0.
\end{equation}
Assume the fractal infinitesimal transformations as follows:
\[
(S_{F}^{\alpha}(x))^* =S_{F}^{\alpha}(x) + \epsilon \xi(x, y), \quad (S_{F}^{\alpha}(y))^* = S_{F}^{\alpha}(y) + \epsilon \phi(x, y),
\]
where \(\epsilon\) is a small parameter, and \(\xi(x, y)\) and \(\phi(x, y)\) are fractal functions. Defining  the fractal infinitesimal generator as:
   \[
   V_{F} = \xi(x, y) D_{F,x}^{\alpha} + \phi(x, y)D_{F,y}^{\alpha}.
   \]
The  prolong the fractal vector field to include first and second fractal derivatives gives:
   \[
   V^{(2)}_{F} = V_{F} + \eta^{(1)}D_{F,y^{(\alpha)}}^{\alpha} + \eta^{(2)} D_{F,y^{(2\alpha)}}^{\alpha},
   \]
   where the coefficients \(\eta^{(1)}\) and \(\eta^{(2)}\) are given by:
 \begin{align}
\eta^{(1)} &= D_{F}^{\alpha}  \phi - y^{(\alpha)} D_{F}^{\alpha}  \xi,\nonumber\\
\eta^{(2)} &= D_{F}^{\alpha}  \eta^{(1)} - y^{(2\alpha)} D_{F}^{\alpha}  \xi.
\end{align}
Here, \(D_{F}^{\alpha}\) is the fractal total derivative with respect to \(x\). By applying the fractal prolonged vector field to Eq.\eqref{o9i00}, we have:
\[
V_{F}^{(2)}\left( D_{F}^{2\alpha}y + y \right) = 0.
\]
This gives:
\[
\xi D_{F,x}^{\alpha}(D_{F}^{2\alpha}y + y) + \phi D_{F,y}^{\alpha} (D_{F}^{2\alpha}y + y) + \eta^{(1)} D_{F,y^{(\alpha)}}^{\alpha}(D_{F}^{2\alpha}y + y) + \eta^{(2)} D_{F,y^{(2\alpha)}}^{\alpha}(D_{F}^{2\alpha}y + y) = 0.
\]
Since \(D_{F,y^{(\alpha)}}^{\alpha}(D_{F}^{2\alpha}y + y) = 0\) and \(D_{F,y^{(2\alpha)}}^{\alpha}(D_{F}^{2\alpha}y + y) = \chi_{F}\), the equation simplifies to:
\[
\eta^{(2)} + \phi = 0.
\]
This gives:
\[
D_{F,x}^{2\alpha} \phi + 2 y^{(\alpha)}D_{F,x,y}^{2\alpha} \phi + (y^{(\alpha)})^2 D_{F,y}^{2\alpha} \phi - y^{(\alpha)} D_{F,x}^{2\alpha}\xi - y^{(\alpha)}D_{F,x,y}^{2\alpha}  \xi - y^{(2\alpha)}D_{F,x}^{\alpha} \xi + \phi = 0.
\]
To solve this system, we separate the equation into terms involving different powers of \( y^{(\alpha)}\), \( y^{(2\alpha)} \), and constants.
\begin{enumerate}
  \item From the term involving \( y^{(2\alpha)} \), we get:
   \[
   D_{F,x}^{\alpha} \xi = 0.
   \]
   This implies \( \xi(x, y) \) is independent of \( x \), so:
   \[
   \xi(x, y) = \xi(y).
   \]
  \item From the terms involving \( y^{(\alpha)} \), we get:
   \[
   D_{F,x,y}^{2\alpha} \phi = 0, \quad D_{F,x}^{2\alpha} \xi= 0.
   \]
   These imply:
   \[
   \phi(x, y) = \phi(x) + A S_{F}^{\alpha}(y).
   \]

  \item Finally, the constant terms give:
   \begin{equation}\label{34oppq}
   D_{F,x}^{2\alpha}\phi(x) + \phi(x) = 0.
   \end{equation}
\end{enumerate}
Eq.\eqref{34oppq} is a second $\alpha$-order linear homogeneous fractal differential equation, whose solution is:
   \[
   \phi(x) = c_3 \cos(S_{F}^{\alpha}(x)) + c_4 \sin(S_{F}^{\alpha}(x)).
   \]
From the symmetry conditions, we find the general solutions:
\[
\xi(x, y) = c_1 S_{F}^{\alpha}(x) + c_2,
\]
\[
\phi(x, y) = c_1 S_{F}^{\alpha}(y) + c_3 \cos(S_{F}^{\alpha}(x)) + c_4 \sin(S_{F}^{\alpha}(x)).
\]
This completes the process of finding the infinitesimals \( \xi(x, y) \) and \( \phi(x, y) \). The fractal symmetry generators corresponding to these solutions are:
  \[
  V_1 = D_{F,x}^{\alpha}, \quad V_2 = S_{F}^{\alpha}(x) D_{F,x}^{\alpha} + S_{F}^{\alpha}(x) D_{F,y}^{\alpha}, \quad V_3 = \cos(S_{F}^{\alpha}(x)) D_{F,y}^{\alpha}, \quad V_4 = \sin(S_{F}^{\alpha}(x)) D_{F,y}^{\alpha}.
  \]
These generators represent \(V_1\) as time translation symmetry, \(V_2\) as scaling symmetry, and \(V_3\) and \(V_4\) as oscillatory symmetries (solutions of the homogeneous fractal differential equation).
As shown in Figure \ref{fig:cantor_set_staircase}, the visualization illustrates the base solution and various symmetry transformations of the Cantor set's integral staircase function \( S_{F}^{\alpha}(x) \).

\begin{figure}[H]
    \centering
    \includegraphics[width=\textwidth]{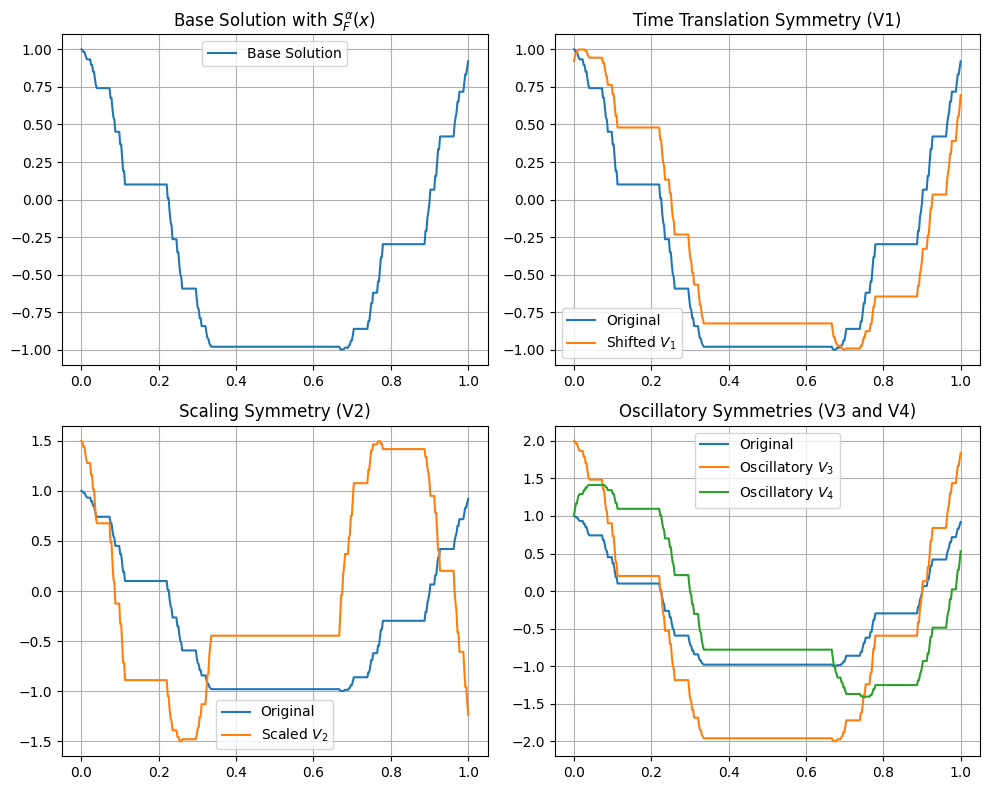}
    \caption{Visualization of the base solution and symmetry transformations of the Cantor set's integral staircase function \( S_{F}^{\alpha}(x) \). The figure illustrates the base solution alongside various symmetry transformations, including time translation, scaling, and oscillatory symmetries.}
    \label{fig:cantor_set_staircase}
\end{figure}

\end{example}

\section{Conclusion \label{S-7}}

In conclusion, this paper has provided a comprehensive examination of fractal calculus, detailing foundational concepts such as Fractal Noether's Theorem, fractal differential forms, vector fields, and Lie derivatives. The introduction of \(k\)-fractal jet space and \(k\)-th fractal prolongation has enriched the theoretical framework necessary for exploring the symmetries of fractal differential equations.

The analysis of both nonlinear and linear \(n\)-th \(\alpha\)-order differential equations demonstrates the intricate relationships between fractal structures and their associated symmetries. By deriving the symmetries of the fractal linear first \(\alpha\)-order differential equation, we have highlighted the applicability of these principles across various mathematical contexts.

The examples illustrated throughout the paper serve not only to reinforce the theoretical concepts discussed but also to showcase the potential for future research in fractal calculus and its applications. As we continue to uncover the complexities of fractal structures, this work contributes to a deeper understanding of how fractality influences mathematical modeling and solutions within diverse fields.\\
\textbf{Conflict of interest:}\\
The authors declare that they do not have any competing interests.\\
\textbf{Dedicated to the memory of Mohammad A. Asadi-Golmankhaneh}



\bibliographystyle{elsarticle-num}
\bibliography{Lilgropufrac}

\begin{thebibliography}{10}
\expandafter\ifx\csname url\endcsname\relax
  \def\url#1{\texttt{#1}}\fi
\expandafter\ifx\csname urlprefix\endcsname\relax\def\urlprefix{URL }\fi
\expandafter\ifx\csname href\endcsname\relax
  \def\href#1#2{#2} \def\path#1{#1}\fi

\bibitem{Mandelbro}
B.~B. Mandelbrot, The Fractal Geometry of Nature, WH freeman, New York, 1982.

\bibitem{Qaswet}
P.~R. Massopust, Fractal Functions, Fractal Surfaces, and Wavelets, Academic Press, Massachusetts, 2017.

\bibitem{falconer1999techniques}
K.~Falconer, Fractal Geometry: Mathematical Foundations and Applications, John Wiley \& Sons, New York, 2004.

\bibitem{feder2013fractals}
J.~Feder, Fractals, Springer Science \& Business Media, New York, 2013.

\bibitem{kigami2001analysis}
J.~Kigami, Analysis on Fractals, Cambridge University Press, Cambridge, 2001.

\bibitem{MR3838440}
P.~E.~T. Jorgensen, Harmonic Analysis, Vol. 128 of CBMS Regional Conference Series in Mathematics, American Mathematical Society, Providence, RI, 2018, smooth and non-smooth, Published for the Conference Board of the Mathematical Sciences.

\bibitem{freiberg2002harmonic}
U.~Freiberg, M.~Z{\"a}hle, Harmonic calculus on fractals-a measure geometric approach {I}, Potential Anal. 16~(3) (2002) 265--277.

\bibitem{Tarasovbook2}
V.~E. Tarasov, Fractional Dynamics, Springer Berlin Heidelberg, New York, 2010.

\bibitem{jorgensen2006analysis}
P.~E. Jorgensen, Analysis and Probability: Wavelets, Signals, Fractals, Vol. 234, Springer Science \& Business Media, New York, 2006.

\bibitem{Withers}
W.~Withers, Fundamental theorems of calculus for {H}ausdorff measures on the real line, J. Math. Anal. Appl. 129~(2) (1988) 581--595.

\bibitem{Barlowqq1}
M.~T. Barlow, E.~A. Perkins, Brownian motion on the sierpinski gasket, Probab. Theory Rel. 79~(4) (1988) 543--623.

\bibitem{stillinger1977axiomaticq}
F.~H. Stillinger, Axiomatic basis for spaces with noninteger dimension, J. Math. Phys. 18~(6) (1977) 1224--1234.

\bibitem{jiang1998some}
H.~Jiang, W.~Su, Some fundamental results of calculus on fractal sets, Commun. Nonlinear Sci. Numer. Simul. 3~(1) (1998) 22--26.

\bibitem{Bongiorno2018derivatives}
D.~Bongiorno, Derivatives not first return integrable on a fractal set, Ric. Mat. 67~(2) (2018) 597--604.

\bibitem{Bongiorno2015fundamental}
D.~Bongiorno, G.~Corrao, On the fundamental theorem of calculus for fractal sets, Fractals 23~(02) (2015) 1550008.

\bibitem{Bongiorno2015integral}
D.~Bongiorno, G.~Corrao, An integral on a complete metric measure space, Real Anal. Exch. 40~(1) (2015) 157--178.

\bibitem{juraev2024fractals}
D.~A. Juraev, N.~M. Mammadzada, Fractals and its applications, Karshi Multidis. Int. Sci. J. 1~(1) (2024) 27--40.

\bibitem{dutkay2006wavelets}
D.~Dutkay, P.~Jorgensen, Wavelets on fractals, Revista Matem{\'a}tica Iberoamericana 22~(1) (2006) 131--180.

\bibitem{jorgensen1998dense}
P.~E. Jorgensen, S.~Pedersen, Dense analytic subspaces in fractal l 2-spaces, J. Anal. Math. 75 (1998) 185--228.

\bibitem{parvate2009calculus}
A.~Parvate, A.~D. Gangal, Calculus on fractal subsets of real line-{I}: Formulation, Fractals 17~(01) (2009) 53--81.

\bibitem{GangalConju}
A.~Parvate, A.~Gangal, Calculus on fractal subsets of real line-{II}: {C}onjugacy with ordinary calculus, Fractals 19~(03) (2011) 271--290.

\bibitem{satin2013fokker}
S.~E. Satin, A.~Parvate, A.~Gangal, Fokker-{P}lanck equation on fractal curves, Chaos Solit. Fractals 52 (2013) 30--35.

\bibitem{Alireza-book}
A.~K. Golmankhaneh, Fractal Calculus and its Applications, World Scientific, Singapore, 2022.

\bibitem{khalili2024fractaldd}
K.~A. Golmankhaneh, K.~Welch, C.~Serpa, P.~E. J{\o}rgensen, Fractal {M}ellin transform and non-local derivatives, Georgian Math. J. 31~(3) (2024) 423--436.

\bibitem{golmankhaneh2016non}
A.~K. Golmankhaneh, D.~Baleanu, Non-local integrals and derivatives on fractal sets with applications, Open Phys. 14~(1) (2016) 542--548.

\bibitem{banchuin2022noise}
R.~Banchuin, Noise analysis of electrical circuits on fractal set, Compel- Int. J. Comput. Math. Electr. Electron. Eng. 41~(5) (2022) 1464--1490.

\bibitem{banchuin20224noise}
R.~Banchuin, Nonlocal fractal calculus based analyses of electrical circuits on fractal set, Compel- Int. J. Comput. Math. Electr. Electron. Eng. 41~(1) (2022) 528--549.

\bibitem{golmankhaneh2018sub}
A.~K. Golmankhaneh, A.~S. Balankin, Sub-and super-diffusion on {C}antor sets: Beyond the paradox, Phys. Lett. A. 382~(14) (2018) 960--967.

\bibitem{Alireza-Fernandez-1}
A.~K. Golmankhaneh, A.~Fernandez, A.~K. Golmankhaneh, D.~Baleanu, Diffusion on middle-$\xi$ {C}antor sets, Entropy 20~(7) (2018) 504.

\bibitem{golmankhaneh2021equilibrium}
A.~K. Golmankhaneh, K.~Welch, Equilibrium and non-equilibrium statistical mechanics with generalized fractal derivatives: A review, Mod. Phys. Lett. A 36~(14) (2021) 2140002.

\bibitem{Tosatti}
L.~Pietronero, E.~Tosatti, Fractals in Physics, Elsevier, Amsterdam, 1986.

\bibitem{e25071008}
A.~Deppman, E.~Meg\'{\i}as, R.~Pasechnik, Fractal derivatives, fractional derivatives and q-deformed calculus, Entropy-switz. 25~(7) (2023) 1008.

\bibitem{Shlesinger-6}
M.~F. Shlesinger, Fractal time in condensed matter, Annu. Rev. Phys. Chem. 39~(1) (1988) 269--290.

\bibitem{Vrobel-3}
S.~Vrobel, Fractal Time, World Scientific, Singapore, 2011.

\bibitem{Welch-5}
K.~Welch, A Fractal Topology of Time: Deepening into Timelessness, Fox Finding Press, Austin, 2020.

\bibitem{nottale1993fractal}
L.~Nottale, Fractal Space-Time and Microphysics: Towards a Theory of Scale Relativity, World Scientific, Singapore, 1993.

\bibitem{khalili2024fractalawee}
K.~A. Golmankhaneh, K.~Welch, C.~Serpa, R.~Rodr{\'\i}guez-L{\'o}pez, Fractal {L}aplace transform: {A}nalyzing fractal curves, J. Anal. 32~(2) (2024) 1111--1137.

\bibitem{khalili2021laplace}
A.~K. Golmankhaneh, S.~M. Nia, Laplace equations on the fractal cubes and {C}asimir effect, Eur. Phys. J. Special Topics 230~(21) (2021) 3895--3900.

\bibitem{golmankhaneh2019sumudu}
A.~K. Golmankhaneh, C.~Tun{\c{c}}, Sumudu transform in fractal calculus, Appl. Math. Comput. 350 (2019) 386--401.

\bibitem{Fourier1ttttttt}
A.~K. Golmankhaneh, K.~Ali, R.~Yilmazer, M.~Kaabar, Local fractal {F}ourier transform and applications, Comput. Methods Differ. Equ. 10~(3) (2021) 595--607.

\bibitem{khalili2019fractalcat}
A.~K. Golmankhaneh, C.~Cattani, Fractal logistic equation, Fractal Fract. 3~(3) (2019) 41.

\bibitem{golmankhaneh2018fractalt}
A.~K. Golmankhaneh, A.~Fernandez, Fractal calculus of functions on {C}antor tartan spaces, Fractal Fract. 2~(4) (2018) 30.

\bibitem{gowrisankar2021fractal}
A.~Gowrisankar, A.~K. Golmankhaneh, C.~Serpa, Fractal calculus on fractal interpolation functions, Fractal Fract. 5~(4) (2021) 157.

\bibitem{khalili2023non}
A.~K. Golmankhaneh, K.~Welch, C.~Serpa, P.~E. J{\o}rgensen, Non-standard analysis for fractal calculus, J. Anal. 31 (2023) 1895--1916.

\bibitem{golmankhaneh2016fractalgag}
A.~K. Golmankhaneh, D.~Baleanu, Fractal calculus involving gauge function, Commun. Nonlinear Sci. Numer. Simul. 37 (2016) 125--130.

\bibitem{khalili2021hyers}
A.~K. Golmankhaneh, C.~Tun{\c{c}}, H.~{\c{S}}evli, Hyers-{U}lam stability on local fractal calculus and radioactive decay, Eur. Phys. J. Special Topics 230~(21) (2021) 3889--3894.

\bibitem{khalili2019analogues}
A.~K. Golmankhaneh, C.~Tun{\c{c}}, Analogues to {L}ie method and {N}oether’s theorem in fractal calculus, Fractal Fract. 3~(2) (2019) 25.

\bibitem{golmankhaneh2023einstein}
A.~K. Golmankhaneh, P.~E. J{\o}rgensen, A.~M. Schlichtinger, Einstein field equations extended to fractal manifolds: A fractal perspective, J. Geom. Phys. 196 (2023) 105081.

\bibitem{golmankhaneh2024expansion}
A.~K. Golmankhaneh, J.~Wanliss, Expansion of the universe on fractal time: A study on the dynamics of cosmic growth, Int. J. Mod. Phys. A 39~(13) (2024) 2450059.

\bibitem{golmankhaneh2023initial}
A.~K. Golmankhaneh, I.~Tejado, H.~{\c{S}}evli, J.~E.~N. Vald{\'e}s, On initial value problems of fractal delay equations, Appl. Math. Comput. 449 (2023) 127980.

\bibitem{KhaliliGolmankhanehWelchSerpaStamova2024}
A.~K. Golmankhaneh, K.~Welch, C.~Serpa, I.~Stamova, Stochastic processes and mean square calculus on fractal curves, Random Oper. Stoch. Equ. 32~(3) (2024) 211--222.

\bibitem{bluman2013symmetries}
G.~W. Bluman, S.~Kumei, Symmetries and differential equations, Vol.~81, Springer Science \& Business Media, 2013.

\bibitem{gonzalez1983symmetries}
F.~Gonz{\'a}lez-Gasc{\'o}n, A.~Gonz{\'a}lez-L{\'o}pez, Symmetries of differential equations. iv, J. Math. Phys. 24~(8) (1983) 2006--2021.

\bibitem{bluman2010applications}
G.~W. Bluman, Applications of symmetry methods to partial differential equations, Springer, 2010.

\bibitem{nucci1997role}
M.~C. Nucci, The role of symmetries in solving differential equations, Math. Comput. Model. 25~(8-9) (1997) 181--193.

\bibitem{grigoriev2010symmetries}
Y.~N. Grigoriev, Symmetries of Integro-Differential Equations, Springer, 2010.

\bibitem{paliathanasis2015symmetries}
A.~Paliathanasis, Symmetries of differential equations and applications in relativistic physics, arXiv preprint arXiv:1501.05129 (2015).

\bibitem{bluman2008symmetry}
G.~Bluman, S.~Anco, Symmetry and integration methods for differential equations, Vol. 154, Springer Science \& Business Media, 2008.

\bibitem{luo2021symmetries}
A.~C. Luo, R.~K. Gazizov, Symmetries and Applications of Differential Equations, Springer, 2021.

\bibitem{gaeta2012nonlinear}
G.~Gaeta, Nonlinear Symmetries and Nonlinear Equations, Springer Science \& Business Media, 2012.

\bibitem{vinogradov2012symmetries}
A.~M. Vinogradov, Symmetries of Partial Differential Equations: Conservation Laws-Applications-Algorithms, Springer Science \& Business Media, 2012.

\bibitem{arrigo2015symmetry}
D.~J. Arrigo, Symmetry analysis of differential equations: an introduction, John Wiley \& Sons, 2015.

\bibitem{hydon2000symmetries}
P.~Hydon, Symmetries and first integrals of ordinary difference equations, Proc. R. Soc. A: Math. Phys. Eng. Sci. 456~(2004) (2000) 2835--2855.

\bibitem{olver1992internal}
P.~J. Olver, Internal symmetries of differential equations, Proc. of Modern Group Analysis: Advanced Analytical and Computational Methods in Mathematical Physics, Acireale, Catania, Italy (1992).

\bibitem{oliveri2004lie}
F.~Oliveri, Lie symmetries of differential equations: direct and inverse problems, Note Mat. 23~(2) (2004) 195--216.

\bibitem{gonzalez1988symmetries}
A.~Gonz{\'a}lez~L{\'o}pez, Symmetries of linear systems of second-order ordinary differential equations, J. Math. Phys. 29~(5) (1988) 1097--1105.

\bibitem{yates2009structural}
J.~W. Yates, N.~D. Evans, M.~J. Chappell, Structural identifiability analysis via symmetries of differential equations, Automatica 45~(11) (2009) 2585--2591.

\bibitem{miller1973symmetries}
W.~Miller, Jr, Symmetries of differential equations. {T}he hypergeometric and {E}uler-{D}arboux equations, SIAM J. Math. Anal. 4~(2) (1973) 314--328.

\bibitem{cicogna2001partial}
G.~Cicogna, G.~Gaeta, Partial {L}ie-point symmetries of differential equations, Phys. A Math. Gen. 34~(3) (2001) 491.

\bibitem{starrett2007solving}
J.~Starrett, Solving differential equations by symmetry groups, The American Mathematical Monthly 114~(9) (2007) 778--792.

\bibitem{gracia2002symmetries}
X.~Gr{\`a}cia, J.~M. Pons, Symmetries and infinitesimal symmetries of singular differential equations, Phys. A Math. Gen. 35~(24) (2002) 5059.

\bibitem{gazizov1998lie}
R.~K. Gazizov, N.~H. Ibragimov, Lie symmetry analysis of differential equations in finance, Nonlinear Dynamics 17 (1998) 387--407.

\bibitem{golmankhaneh2025fractalvar}
A.~K. Golmankhaneh, C.~Cattani, R.~Pasechnik, S.~Furuichi, P.~E. Jorgensen, Fractal calculus of variations for problems with constraints, Modern Physics Letters A (2025) 2550001.

\end{thebibliography}

\end{document}